\documentclass[preprint,12pt]{elsarticle}

\biboptions{numbers,sort&compress}
\usepackage{amssymb}
\usepackage{amsthm}
\usepackage{amsmath}
\usepackage{epic}
\usepackage{CJK}
\usepackage{setspace}
\usepackage{bm}
\newtheorem{thm}{Theorem}[section]
\newtheorem{cor}[thm]{Corollary}
\newtheorem{lem}[thm]{Lemma}

\newtheorem{example}[thm]{Example}
\newtheorem{rem}{Remark}

\journal{}

\begin{document}
\begin{spacing}{1.15}
\begin{CJK*}{GBK}{song}
\begin{frontmatter}
\title{\textbf{Spectral bounds for the independence number of graphs and even uniform hypergraphs}}

\author{Xinyu Hu}
\author{Jiang Zhou}\ead{zhoujiang@hrbeu.edu.cn}
\author{Changjiang Bu}

\address{College of Mathematical Sciences, Harbin Engineering University, Harbin 150001, PR China}

\begin{abstract}
In this paper, we give spectral upper bounds for the independence number of even uniform hypergraphs and graphs, extend the Hoffman bound to even uniform hypergraphs, and give a simple spectral condition for determining the independence number, the Shannon capacity and the Lov\'{a}sz number of a graph. The Hoffman bound on the Lov\'{a}sz number is also extended from regular graphs to general graphs.
\end{abstract}

\end{frontmatter}

\section{Introduction}
For a graph $G$ with vertex set $V(G)$ and edge set $E(G)$, the \textit{adjacency matrix} $A_G$ of $G$ is a $|V(G)|\times|V(G)|$ symmetric matrix with entries
\begin{eqnarray*}
(A_G)_{ij}=\begin{cases}1~~~~~~~~~~\mbox{if}~\{i,j\}\in E(G),\\
0~~~~~~~~~~\mbox{if}~\{i,j\}\notin E(G).\end{cases}
\end{eqnarray*}
Eigenvalues of the adjacency matrix of $G$ are called \textit{eigenvalues} of $G$.

The \textit{independence number} $\alpha(G)$ of graph $G$ is the maximum size of independent sets in $G$. The following is the celebrated Hoffman bound on $\alpha(G)$.
\begin{thm}\label{Hoffman}
\textup{\cite{Brouwer}} Let $G$ be an $n$-vertex $d$-regular ($d>0$) graph with minimum eigenvalue $\lambda$. Then
\begin{eqnarray*}
\alpha(G) \leqslant\frac{-\lambda n}{{d - {\lambda }}}.
\end{eqnarray*}
If an independent set $S$ meets this bound, then every vertex not in $S$ is adjacent to exactly $-\lambda$ vertices of $S$.
\end{thm}

Haemers generalized the Hoffman bound to the general case.
\begin{thm}\label{thm1.2}
\textup{\cite[Theorem 3.3]{Haemers}} Let $G$ be an $n$-vertex graph with minimum degree $\delta>0$, and let $\lambda_1$ and $\lambda_n$ be the maximum and the minimum eigenvalue of $G$, respectively. Then
\begin{eqnarray*}
\alpha(G)\leq n\frac{-\lambda_1\lambda_n}{\delta^2-\lambda_1\lambda_n}.
\end{eqnarray*}
\end{thm}
The \textit{Laplacian matrix} of a graph $G$ is defined as $L_G=D_G-A_G$, where $D_G$ is the diagonal matrix of vertex degrees of $G$. Eigenvalues of $L_G$ are called the \textit{Laplacian eigenvalues} of $G$. The following is another generalization of Theorem \ref{Hoffman}.
\begin{thm}\label{thm1.3}
\textup{\cite[Corollary 3.6]{Godsil}} Let $G$ be an $n$-vertex graph with the largest Laplacian eigenvalue $\mu>0$ and the minimum degree $\delta$. Then
\begin{eqnarray*}
\alpha(G)\leq\frac{n(\mu-\delta)}{\mu}.
\end{eqnarray*}
\end{thm}
There are various Hoffman-type ratio bounds, which are useful in graph theory, coding theory and extremal combinatorics \cite{Delsarte,EllisFilmus,Ellis,Godsil,Haemers}.

For a graph $G=(V(G),E(G))$, let $G^k$ denote the graph whose vertex set is $V(G)^k$, in which two vertices $u_1\cdots u_k$ and $v_1\cdots v_k$ are adjacent if and only if for each $i\in\{1,\ldots,k\}$ either $u_i=v_i$ or $\{u_i,v_i\}\in E(G)$. The \textit{Shannon capacity} \cite{Alon,shannon1} of $G$ is defined as $\Theta(G)=\sup_{k}\alpha(G^k)^{1/k}$, which is the effective size of an alphabet in an information channel represented by the graph $G$. In \cite{L}, the Lov\'{a}sz number $\vartheta(G)$ was used to study the Shannon capacity of graphs. Lov\'{a}sz proved that $\alpha(G)\leq\Theta(G)\leq\vartheta(G)$, and posed the problem of finding graphs with $\Theta(G)=\vartheta(G)$ \cite[Problem 1]{L}. For a regular graph $G$, the Hoffman bound is also an upper bound for $\vartheta(G)$ (see \cite[Theorem 9]{L}).

Let $\mathcal{H}$ be a $k$-uniform hypergraph, and suppose that $0<t<k$. For a vertex subset $S\subseteq V(\mathcal{H})$, if $|e\cap S|\in\{0,t\}$ for each edge $e\in E(\mathcal{H})$, then we say that $S$ is a \textit{$t$-independent set} of $\mathcal{H}$. The \textit{t-independence number} $\alpha_t(\mathcal{H})$ of $\mathcal{H}$ is the maximum size of $t$-independent sets in $\mathcal{H}$. Then $\alpha_1(\mathcal{H})$ equals to the strongly independence number \cite{ZhouLi} of $\mathcal{H}$. When $\mathcal{H}$ is a graph ($k=2$), $\alpha_1(\mathcal{H})$ equals to the independence number $\alpha(\mathcal{H})$. So the $t$-independence number of uniform hypergraphs is a generalization of the independence number of graphs.

For the spectral ratio-type bound on hypergraphs, as far as we currently know, the upper bound in Theorem \ref{thm1.3} have been extended to the independence number of a simplicial complex by using the eigenvalues of the upper Laplacian operators (see \cite{Bachoc,Golubev}). The contribution of this paper is to extend the spectral ratio-type bound to hypergraphs within a different framework, that is, establish Hoffman-type bounds by using adjacency tensor eigenvalues of hypergraphs. There is a natural correspondence between hypergraphs and tensors. Recent years, the research on spectral theory of hypergraphs via tensors has attracted extensive attention \cite{Abiad,Chen,Clark,Cooper,Ellingham,Fan,Gao,LiMohar,LiSu}. It is natural and interesting to study the relationship between the independence number and tensor eigenvalues of hypergraphs. We will give Hoffman-type bounds on even uniform hypergraphs in terms of tensor eigenvalues.

The paper is organized as follows. In Section 2, we introduce some notations and auxiliary lemmas. In Section 3, we give spectral upper bounds for the $t$-independence number of even uniform hypergraphs, which extends the Hoffman bound to even uniform hypergraphs. In Section 4, we obtain spectral upper bounds for the independence number and the Lov\'{a}sz number of graphs based on the results in Section 3, and give a spectral condition for determining the independence number, the Shannon capacity and the Lov\'{a}sz number of a graph. The Hoffman bound on the Lov\'{a}sz number is also extended from regular graphs to general graphs.

\section{Preliminaries}

\subsection{Lov\'{a}sz number}
We now introduce the Lov\'{a}sz number defined in \cite{L}. An \textit{orthonormal representation} of an $n$-vertex graph $G$ is a set $\{u_1,\ldots,u_n\}$ of unit real vectors such that $u_i^\top u_j=0$ if $i$ and $j$ are two nonadjacent vertices in $G$. The \textit{value} of an orthonormal representation $\{u_1,\ldots,u_n\}$ is defined to be
\begin{eqnarray*}
\min_c\max_{1\leq i\leq n}\frac{1}{(c^\top u_i)^2},
\end{eqnarray*}
where $c$ ranges over all unit real vectors. The \textit{Lov\'{a}sz number} $\vartheta(G)$ is the minimum value of all orthonormal representations of $G$.

For any graph $G$, the independence number $\alpha(G)$, the Shannon capacity $\Theta(G)$ and the Lov\'{a}sz number $\vartheta(G)$ have the following relations.
\begin{lem}\label{Shannon}
\textup{\cite{L}} For any graph $G$, we have
\begin{eqnarray*}
\alpha (G) \leqslant \Theta \left( G \right) \leqslant \vartheta (G).
\end{eqnarray*}
\end{lem}

For a real square matrix $M$, the \textit{group inverse} of $M$, denoted by $M^\#$, is the matrix $X$ such that $MXM=M,~XMX=X$ and $MX=XM$. It is known \cite{Ben-Israel} that $M^\#$ exists if and only if $\mbox{\rm rank}(M)=\mbox{\rm rank}(M^2)$. If $M^\#$ exists, then $M^\#$ is unique.
\begin{lem} \label{zhou1}
\textup{\cite{zhou}} For any graph $G$, we have
\begin{eqnarray*}
\alpha(G)\leq\min_{(M,x)\in\mathcal{M}(G)}x^\top M^\#x\max_{u\in V(G)}\frac{(M)_{uu}}{x_u^2}=\vartheta(G),
\end{eqnarray*}
where $\mathcal{M}(G)$ denotes the set of real matrix-vector pairs $(M,x)$ satisfying:

\noindent (a) $M$ is a positive semidefinite matrix indexed by $V(G)$ such that $(M)_{ij}=0$ if $i,j$ are two nonadjacent distinct vertices.

\noindent \noindent (b) $x\in R(M)=\{z:z=My,y\in\mathbb{R}^n\}$ {\rm (}$n=|V(G)|${\rm )} and $x_u\neq0$ for each $u\in V(G)$.
\end{lem}

\subsection{Spectrum of hypergraphs}
An order $k$ dimension $n$ complex tensor $\mathcal{A}=(a_{i_1i_2\cdots i_k})$ is a multidimensional array with $n^k$ entries, where $i_{j}=1,2,...,n$, $j=1,2,...,k$. For a vector $x=(x_1,\ldots,x_n)^\top\in\mathbb{C}^{n}$, let $\mathcal{A}x^k=\sum_{i_1,\ldots,i_k=1}^na_{i_1i_2\cdots i_k}x_{i_1}\cdots x_{i_k}$, and let $\mathcal{A}x^{k-1}$ denote the vector in $\mathbb{C}^{n}$ whose $i$-th component is
\[{\left( {\mathcal{A}{x^{k - 1}}} \right)_i} = \sum\limits_{{i_2},{i_3}, \cdots ,{i_k} = 1}^n {{a_{i{i_2}{i_3} \cdots {i_k}}}} {x_{{i_2}}}{x_{{i_3}}} \cdots {x_{{i_k}}}.\]
If there exist $\lambda\in\mathbb{C}$ and nonzero vector $x=(x_1,\ldots,x_n)^\top\in\mathbb{C}^{n}$ satisfying
\begin{align*}
\mathcal{A}x^{k-1}=\lambda x^{[k-1]},
\end{align*}
then $\lambda$ is called an eigenvalue \cite{Lim,Qi} of $\mathcal{A}$, and $x$ is called an eigenvector of $\mathcal{A}$ corresponding to $\lambda$, where $x^{[k-1]}=(x_{1}^{k-1},\ldots,x_{n}^{k-1})^\top$. If $\lambda$ is a real eigenvalue of $\mathcal{A}$ with a real eigenvector, then $\lambda$ is called an $H$-eigenvalue \cite{Qi} of $\mathcal{A}$.

A hypergraph $\mathcal{H}$ is called $k$-\textit{uniform} if each edge of $\mathcal{H}$ contains exactly $k$ distinct vertices. Let $V(\mathcal{H})$ and $E(\mathcal{H})$ denote the vertex set and the edge set of $\mathcal{H}$, respectively. The \textit{adjacency tensor} \cite{Cooper} of a $k$-uniform hypergraph $\mathcal{H}$, denoted by $\mathcal{A}_{\mathcal{H}}=(a_{i_1i_2\cdots i_k})$, is an order $k$ dimension $|V(\mathcal{H})|$ tensor with entries
\begin{eqnarray*}
a_{i_1i_2\cdots i_k}=\begin{cases}\frac{1}{(k-1)!}~~~~~~~\mbox{if}~\{i_1,i_2,\ldots,i_k\}\in E(\mathcal{H}),\\
0~~~~~~~~~~~~~\mbox{otherwise}.\end{cases}
\end{eqnarray*}
The minimum $H$-eigenvalue of $\mathcal{A}_{\mathcal{H}}$ is called the minimum $H$-eigenvalue of $\mathcal{H}$.

Let $\lambda$ be the minimum $H$-eigenvalue of $\mathcal{H}$, and let $\mathcal{I}$ denote the diagonal tensor such that $(\mathcal{I})_{i\cdots i}=1$ for all $i$. When $k$ is even, we have
\begin{eqnarray*}
\lambda&=&\min\left\{\mathcal{A}_{\mathcal{H}}x^k:x\in{\mathbb{R}^n},\sum_{i\in V(\mathcal{H})}x_i^k=1\right\}\\
&=&k\min\left\{\sum_{\{ i_1,i_2,\ldots,i_k\}\in E(\mathcal{H})}x_{i_1}x_{i_2}\cdots x_{i_k}:x\in{\mathbb{R}^n},\sum_{i\in V(\mathcal{H})}x_i^k=1\right\}.
\end{eqnarray*}
Hence $\mathcal{A}_{\mathcal{H}}-\lambda\mathcal{I}$ is positive semidefinite, that is, $(\mathcal{A}_{\mathcal{H}}-\lambda\mathcal{I})x^k\geq0$ for any real vector $x$. For any nonzero real vector $x$, $(\mathcal{A}_{\mathcal{H}}-\lambda\mathcal{I})x^k=0$ if and only if $x$ is an eigenvector of $\lambda$. If $k=2$ then $\lambda$ is exactly the minimum eigenvalue of an ordinary graph.

\section{Hoffman bound for even uniform hypergraphs}
For a vertex $i$ in a hypergraph $\mathcal{H}$, let $E(i)$ denote the set of edges containing $i$. The degree $d_i$ of $i$ is the number of edges containing $i$, that is, $d_i=|E(i)|$.

We first give the following Hoffman-type bound for even uniform hypergraphs.
\begin{thm}\label{t}
Suppose that $t$ is an odd integer such that $0<t<k$ for an even integer $k$. Let $\mathcal{H}$ be a $k$-uniform hypergraph with $n$ vertices and $m\geq1$ edges, and let $\delta$ and $\lambda$ be the minimum degree and the minimum $H$-eigenvalue of $\mathcal{H}$, respectively. Then
\begin{eqnarray*}
{\alpha _t}(\mathcal{H}) \leqslant \frac{{t\left( {km - n \lambda } \right){{\left( { - \lambda } \right)}^{\frac{t}{{k - t}}}}}}{{\left( {k - t} \right){\delta ^{\frac{k}{{k - t}}}} + \left( {k \delta  - t\lambda } \right){{\left( { - \lambda } \right)}^{\frac{t}{{k - t}}}}}},
\end{eqnarray*}
with equality if and only if there exits a $t$-independent set $S$ such that every vertex in $S$ has degree $\delta$, and every vertex $i$ outside $S$ satisfies
\begin{eqnarray*}
\left| {\{ e:e \in E(i),e\cap S\neq\emptyset \} } \right| = \frac{{{{\left( { - \lambda } \right)}^{\frac{k}{{k - t}}}} + {d_i}{{\left( { - \lambda } \right)}^{\frac{t}{{k - t}}}}}}{{{\delta ^{\frac{t}{{k - t}}}} + {{\left( { - \lambda } \right)}^{\frac{t}{{k - t}}}}}}.
\end{eqnarray*}
\end{thm}
\begin{proof}
Since $m\geq1$, we have $\lambda<0$. Let $S$ be a $t$-independent set of $\mathcal{H}$ such that $|S|=\alpha_t(\mathcal{H})$. For $c\geq0$, take $x=(x_1,x_2,\ldots,x_n)^\top$ satisfying
$${x_i} =
\begin{cases}
c & i \in S,\\
- 1 & i \notin S.
\end{cases}
$$
Let $\mathcal{A}$ be the adjacency tensor of $\mathcal{H}$. Since $\mathcal{A}-\lambda \mathcal{I}$ is a positive semidefinite tensor, we have
\begin{eqnarray*}
(\mathcal{A}-\lambda \mathcal{I})x^k=-\lambda\sum_{i=1}^nx_i^{k}+k\sum_{\{ i_1,i_2,\ldots,i_k\}\in E}x_{i_1}x_{i_2}\cdots x_{i_k}\geq0,
\end{eqnarray*}
where $E=E(\mathcal{H})$ is the edge set of $\mathcal{H}$.
Since $|S|=\alpha_t(\mathcal{H})$ and $|e\cap S|\in\{0,t\}$ for each edge $e\in E$, we have

\[( {\mathcal{A}}-\lambda \mathcal{I} )x^k = -\lambda\sum_{i=1}^nx_i^k  +k
 \sum_{\substack{\{i_1,i_2,\ldots,i_k\}\in E\\ |\{i_1,i_2,\ldots,i_k\}\cap S|=t}}
{{x_{{i_1}}}{x_{{i_2}}} \cdots {x_{{i_k}}}}  +k
\sum_{\substack{\{i_1,i_2,\ldots,i_k\} \in E\\ \{i_1,i_2,\ldots,i_k\} \cap S = \emptyset}} {{x_{{i_1}}}{x_{{i_2}}} \cdots {x_{{i_k}}}} \]

\[= - {\lambda} \alpha_{t}(\mathcal{H}) {c^k} - \lambda (n - {\alpha_{t}}(\mathcal{H}) ) - kc^tt^{-1}\sum_{i\in S}d_i+ k\left( {m - t^{-1}\sum_{i\in S}d_i}\right) \geqslant 0.\]
Then
\[{\alpha _t}(\mathcal{H})\left( {-\lambda {c^k} + \lambda } \right) + km -n \lambda  \geqslant \left( {k{c^t} + k} \right)t^{-1}\sum_{i\in S}d_i \geqslant \left( {k{c^t} + k} \right)\frac{{{\alpha _t}(\mathcal{H})\delta }}{t},\]

\[\left( {-t \lambda \left( {1 - {c^k}} \right) + k \delta \left( {1 + {c^t}} \right)} \right){\alpha _t}(\mathcal{H})\leq t\left( {km - n\lambda } \right).\]
Take $c = {{\left( \frac{\delta }{-\lambda } \right)}^{\frac{1}{k-t}}}$, then
$ -t \lambda \left( {1 - {c}^k} \right) + k \delta \left( {1 + {c}^t} \right) = \left( {k - t} \right)\delta {c}^t + k \delta - t \lambda  > 0$ and
\[{\alpha _t}(\mathcal{H}) \leqslant \frac{{t\left( {km - n \lambda } \right)}}{{\left( {k - t} \right)\delta {c^t} + k \delta  - t \lambda }}=\frac{{t\left( {km - n \lambda } \right){{\left( { - \lambda } \right)}^{\frac{t}{{k - t}}}}}}{{\left( {k - t} \right){\delta ^{\frac{k}{{k - t}}}} + \left( {k \delta  - t\lambda } \right){{\left( { - \lambda } \right)}^{\frac{t}{{k - t}}}}}},\]
with equality if and only if every vertex in $S$ has degree $\delta$, and $x$ is an eigenvector of $\lambda$. Recall that $x$ is a vector satisfying
$${x_i} =
\begin{cases}
 {{\left( \frac{\delta }{-\lambda } \right)}^{\frac{1}{k-t}}} & i \in S,\\
-1 & i \notin S.
\end{cases}
$$
When every vertex in $S$ has degree $\delta$, $\mathcal{A}x^{k-1}=\lambda x^{[k-1]}$ is equivalent to
\begin{eqnarray*}
-\lambda=\lambda x_i^{k-1}=\sum_{\{i,i_2,\ldots,i_k\}\in E}x_{i_2}\cdots x_{i_k},~i\notin S.
\end{eqnarray*}
For $i\notin S$, let $d'_i = \left| \{e: e \in E(i),e\cap S \neq \emptyset \}\right|$. Then the above equation is equivalent to
$$-\lambda = d'_i {{\left( \frac{\delta }{-\lambda } \right)}^{\frac{t}{k-t}}} - (d_i - d'_i),$$
that is,
$$d'_i = \left| {\{ e:e \in E(i),e\cap S \ne \emptyset \} } \right| = \frac{{{{\left( { - \lambda } \right)}^{\frac{k}{{k - t}}}} + {d_i}{{\left( { - \lambda } \right)}^{\frac{t}{{k - t}}}}}}{{{\delta ^{\frac{t}{{k - t}}}} + {{\left( { - \lambda } \right)}^{\frac{t}{{k - t}}}}}}.$$
\end{proof}
The following is an example for the equality case of Theorem \ref{t}.
\begin{example}\label{z}
Suppose that $\mathcal{H}$ is a $k$-uniform hypergraph with a partition $V(\mathcal{H})=V_1\cup V_2$ such that $|e\cap V_1|=t$ and $|e\cap V_2|=k-t$ are both odd for each $e\in E(\mathcal{H})$. If $\mathcal{H}$ is $d$-regular, then $d$ is the spectral radius of $\mathcal{H}$, and $-d$ is the minimum $H$-eigenvalue of $\mathcal{H}$ (see \cite[Theorem 2.3]{Shao}). In this case, $V_1$ is a $t$-independent set satisfying conditions in Theorem \ref{t}.
\end{example}
Take $t=1$ in Theorem \ref{t}, we get the following bound for the strongly independence number $\alpha_1(\mathcal{H})$.
\begin{cor}\label{1}
Let $\mathcal{H}$ be a $k$-uniform hypergraph with $n$ vertices and $m\geq1$ edges, and let $\delta$ and $\lambda$ be the minimum degree and the minimum $H$-eigenvalue of $\mathcal{H}$, respectively. If $k$ is even, then
\[{\alpha}_1(\mathcal{H}) \leqslant \frac{{\left( {km - n \lambda } \right){{\left( { - \lambda } \right)}^{\frac{1}{{k - 1}}}}}}{{\left( {k - 1} \right){\delta ^{\frac{k}{{k - 1}}}} + \left( {k \delta  - \lambda } \right){{\left( { - \lambda } \right)}^{\frac{1}{{k - 1}}}}}},\]
with equality if and only if there exits a $1$-independent set $S$ such that every vertex in $S$ has degree $\delta$, and every vertex $i$ outside $S$ satisfies
\begin{eqnarray*}
\left| {\{ e:e \in E(i),e\cap S \ne \emptyset \} } \right| = \frac{{{{\left( { - \lambda } \right)}^{\frac{k}{{k - 1}}}} + {d_i}{{\left( { - \lambda } \right)}^{\frac{1}{{k - 1}}}}}}{{{\delta ^{\frac{1}{{k - 1}}}} + {{\left( { - \lambda } \right)}^{\frac{1}{{k - 1}}}}}}.
\end{eqnarray*}
\end{cor}

\begin{rem}\label{d}
If $\mathcal{H}=G$ is a $d$-regular graph in Theorem \ref{t} ($k=2,\delta=d$), then the Hoffman bound $\alpha(G) \leqslant\frac{-\lambda n}{{d - {\lambda }}}$ follows from the upper bound in Theorem \ref{t}, and the equality holds if and only if there exits an independent set $S$ such that every vertex not in $S$ is adjacent to exactly $-\lambda$ vertices of $S$.
\end{rem}
Let $\mathcal{H}$ be a $k$-uniform hypergraph with $n$ vertices. We say that a real tensor $\mathcal{A}=(a_{i_1i_2\cdots i_k})$ of dimension $n$ is a \textit{signed adjacency tensor} of $\mathcal{H}$ if
\begin{eqnarray*}
|a_{i_1i_2\cdots i_k}|=\begin{cases}\frac{1}{(k-1)!}~~~~~~~~~\mbox{if}~\{i_1,i_2,\ldots,i_k\}\in E(\mathcal{H}),\\
0~~~~~~~~~~~~~~~\mbox{otherwise}.\end{cases}
\end{eqnarray*}
Let $\mathcal{S}(\mathcal{H})$ denote the set of signed adjacency tensors of $\mathcal{H}$.
\begin{thm}
Suppose that $t$ is an even integer such that $0<t<k$ for an even integer $k$. Let $\mathcal{H}$ be a $k$-uniform hypergraph with $n$ vertices and $m\geq1$ edges, and let $\delta$ be the minimum degree of $\mathcal{H}$. Then there exists $\mathcal{A}\in\mathcal{S}(\mathcal{H})$ such that
\begin{eqnarray*}
{\alpha _t}(\mathcal{H}) \leqslant \frac{{t\left( {km - n \lambda } \right){{\left( { - \lambda } \right)}^{\frac{t}{{k - t}}}}}}{{\left( {k - t} \right){\delta ^{\frac{k}{{k - t}}}} + \left( {k \delta  - t\lambda } \right){{\left( { - \lambda } \right)}^{\frac{t}{{k - t}}}}}},
\end{eqnarray*}
where $\lambda$ is the minimum $H$-eigenvalue of $\mathcal{A}$.
\end{thm}
\begin{proof}
Let $S$ be a $t$-independent set of $\mathcal{H}$ such that $|S|=\alpha_t(\mathcal{H})$. For $c\geq0$, take $x=(x_1,x_2,\ldots,x_n)^\top$ satisfying
$${x_i} =
\begin{cases}
c & i \in S,\\
1 & i \notin S.
\end{cases}
$$
We can choose $\mathcal{A}=(a_{i_1i_2\cdots i_k})\in\mathcal{S}(\mathcal{H})$ such that
\begin{eqnarray*}
a_{i_1i_2\cdots i_k}=\begin{cases}\frac{1}{(k-1)!}~~~~~~~~~\mbox{if}~\{i_1,i_2,\ldots,i_k\}\in E(\mathcal{H})~\mbox{and}~|\{i_1,i_2,\ldots,i_k\}\cap S|=0,\\
-\frac{1}{(k-1)!}~~~~~~~\mbox{if}~\{i_1,i_2,\ldots,i_k\}\in E(\mathcal{H})~\mbox{and}~|\{i_1,i_2,\ldots,i_k\}\cap S|=t,\\
0~~~~~~~~~~~~~~~\mbox{otherwise}.\end{cases}
\end{eqnarray*}
Let $\lambda$ be the minimum $H$-eigenvalue of $\mathcal{A}$, then $\mathcal{A}-\lambda \mathcal{I}$ is a positive semidefinite tensor. Since $|S|=\alpha_t(\mathcal{H})$ and $|e\cap S|\in\{0,t\}$ for each edge $e\in E(\mathcal{H})$, we have

\[( {\mathcal{A}}-\lambda \mathcal{I} )x^k = -\lambda\sum_{i=1}^nx_i^k-k
 \sum_{\substack{\{i_1,i_2,\ldots,i_k\}\in E(\mathcal{H})\\ |\{i_1,i_2,\ldots,i_k\}\cap S|=t}}
{{x_{{i_1}}}{x_{{i_2}}} \cdots {x_{{i_k}}}}  +k
\sum_{\substack{\{i_1,i_2,\ldots,i_k\} \in E(\mathcal{H})\\ \{i_1,i_2,\ldots,i_k\} \cap S = \emptyset}} {{x_{{i_1}}}{x_{{i_2}}} \cdots {x_{{i_k}}}} \]

\[= - {\lambda} \alpha_{t}(\mathcal{H}) {c^k} - \lambda (n - {\alpha_{t}}(\mathcal{H}) ) - kc^tt^{-1}\sum_{i\in S}d_i+ k\left( {m - t^{-1}\sum_{i\in S}d_i}\right) \geqslant 0.\]
Then
\[{\alpha _t}(\mathcal{H})\left( {-\lambda {c^k} + \lambda } \right) + km -n \lambda  \geqslant \left( {k{c^t} + k} \right)t^{-1}\sum_{i\in S}d_i \geqslant \left( {k{c^t} + k} \right)\frac{{{\alpha _t}(\mathcal{H})\delta }}{t},\]

\[\left( {-t \lambda \left( {1 - {c^k}} \right) + k \delta \left( {1 + {c^t}} \right)} \right){\alpha _t}(\mathcal{H})\leq t\left( {km - n\lambda } \right).\]
Take $c = {{\left( \frac{\delta }{-\lambda } \right)}^{\frac{1}{k-t}}}$, then
$ -t \lambda \left( {1 - {c}^k} \right) + k \delta \left( {1 + {c}^t} \right) = \left( {k - t} \right)\delta {c}^t + k \delta - t \lambda  > 0$ and
\[{\alpha _t}(\mathcal{H}) \leqslant \frac{{t\left( {km - n \lambda } \right)}}{{\left( {k - t} \right)\delta {c^t} + k \delta  - t \lambda }}=\frac{{t\left( {km - n \lambda } \right){{\left( { - \lambda } \right)}^{\frac{t}{{k - t}}}}}}{{\left( {k - t} \right){\delta ^{\frac{k}{{k - t}}}} + \left( {k \delta  - t\lambda } \right){{\left( { - \lambda } \right)}^{\frac{t}{{k - t}}}}}}.\]
\end{proof}

\section{Independence number and Lov\'{a}sz number of graphs}
Take $k=2$ in Theorem \ref{t}, we get the following spectral bound for the independence number of a graph.
\begin{cor}\label{2}
Let $G$ be an $n$-vertex graph with minimum degree $\delta$ and average degree $\overline{d}>0$, and let $\lambda$ be the minimum eigenvalue of $G$. Then
$$\alpha(G)\leq\frac{-\lambda n(\overline{d} - \lambda)}{(\delta - \lambda)^{2}},$$
with equality if and only if there exists an independent set $S$ such that every vertex in $S$ has degree $\delta$, and every vertex $i$ not in $S$ is adjacent to exactly $\frac{-\lambda(d_i - \lambda)}{\delta - \lambda}$ vertices of $S$.
\end{cor}
In the necessary and sufficient conditions of Corollary \ref{2}, we always have
\begin{eqnarray*}
\frac{-\lambda(d_i-\lambda)}{\delta-\lambda}\geq-\lambda.
\end{eqnarray*}
Motivated by this fact, we give the following spectral condition for determining the independence number, the Shannon capacity and the Lov\'{a}sz number of a graph.
\begin{thm}\label{theta2}
Let $G$ be an $n$-vertex graph with minimum eigenvalue $\lambda<0$, and let $S$ be an independent set such that every vertex $i\notin S$ satisfies
\begin{eqnarray*}
|\{j:\{i,j\}\in E(G),j\in S\}|\geq-\lambda.
\end{eqnarray*}
Then
\begin{eqnarray*}
\alpha(G) = \Theta(G) = \vartheta(G) = |S|.
\end{eqnarray*}
\end{thm}
\begin{proof}
Let $A$ be the adjacency matrix of $G$. Then $M=A-\lambda I$ is positive semidefinite, where $I$ is the identity matrix. Let $y$ be the indicator vector for $S$, that is
$${y_i} =
\begin{cases}
1 & i \in S,\\
0 & i \notin S.
\end{cases}
$$
Since every vertex $i\notin S$ satisfies $|\{j:\{i,j\}\in E(G),j\in S\}|\geq-\lambda>0$, we know that $x=My$ is a positive vector and
\begin{eqnarray*}
\min_{u\in V(G)}x_u=-\lambda.
\end{eqnarray*}
So we have
\[x^\top M^\#x\max_{u\in V(G)}\frac{(M)_{uu}}{x_u^2}=y^\top M{M^\# }My\mathop {\max }\limits_{u \in V\left( G \right)} \frac{{{{\left( M \right)}_{uu}}}}{{x_u^2}} =\frac{1}{-\lambda} {y^\top}My= |S|.\]
By Lemma \ref{zhou1}, we have
$$\alpha(G) \leq \vartheta(G)\leq x^\top M^\#x\max_{u\in V(G)}\frac{(M)_{uu}}{x_u^2}=  |S|.$$
Since $|S|\leq\alpha(G)$, by Lemma \ref{Shannon}, we have
\[\alpha (G) = \Theta \left( G \right) = \vartheta (G) = |S|.\]
\end{proof}
A regular graph attaining the Hoffman bound must have an independent set $S$ satisfying the condition in Theorem \ref{theta2} (see Theorem \ref{Hoffman}). The following is an nonregular example for Theorem \ref{theta2}.
\begin{example}
Let $G$ be a graph with vertex set $\{1,2,\ldots,n\}$ and minimum eigenvalue $\lambda_0$, and let $H=G(p_1,\ldots,p_n)$ be the graph obtained from $G$ by attaching $p_i$ pendant vertices at vertex $i$, where $p_1\geq\cdots\geq p_n\geq1$ and $\lambda_0+p_n-p_n^{-1}p_1\geq0$. Then
\begin{eqnarray*}
\alpha(H)=\Theta(H)=\vartheta(H)=p_1+\cdots+p_n.
\end{eqnarray*}
\end{example}
\begin{proof}
The adjacency matrix of $H$ can be written as $A_H=\begin{pmatrix}0&B\\B^\top&A_G\end{pmatrix}$, where $A_G$ is the adjacency matrix of $G$ and $B^\top B={\rm diag}(p_1,\ldots,p_n)$. Notice that
\begin{eqnarray*}
A_H+p_nI=\begin{pmatrix}p_nI&B\\B^\top&A_G+p_nI\end{pmatrix}
\end{eqnarray*}
is positive semidefinite, because the Schur complement $A_G+p_nI-p_n^{-1}B^\top B$ is positive semidefinite. Hence $p_n\geq-\lambda$, where $\lambda$ is the minimum eigenvalue of $H$. All new added pendant vertices form an independent set $S$ in $H$, and $S$ satisfies the condition in Theorem \ref{theta2}. Hence
\begin{eqnarray*}
\alpha(H)=\Theta(H)=\vartheta(H)=|S|=p_1+\cdots+p_n.
\end{eqnarray*}
\end{proof}
Let $G$ be an $n$-vertex $d$-regular graph with minimum eigenvalue $\lambda$. In \cite[Theorem 9]{L}, Lov\'{a}sz proved that
\begin{eqnarray*}
\vartheta(G)\leq\frac{-\lambda n}{d-\lambda}.
\end{eqnarray*}
We next extend this result to general graphs, that is, the upper bound in Corollary \ref{2} is also an upper bound for $\vartheta(G)$.
\begin{thm}\label{theta1}
Let $G$ be an $n$-vertex graph with minimum degree $\delta$ and average degree $\overline{d}>0$, and let $\lambda$ be the minimum eigenvalue of $G$. Then
$$\alpha(G) \leq \vartheta(G) \leq \frac{-\lambda n(\overline{d} - \lambda)}{(\delta - \lambda)^{2}}.$$
\end{thm}
\begin{proof}
Let $A$ be the adjacency matrix of $G$. Then $M=A-\lambda I$ is positive semidefinite, where $I$ is the identity matrix. Let $x= M\mathrm{\mathbf{1}}$, where $\mathbf{1}$ is the all-ones vector. Then
$$x^\top M^\#x\max_{u\in V(G)}\frac{(M)_{uu}}{x_u^2}=\mathbf{1}^\top MM^\#M{\mathbf{1}}\mathop {\max }\limits_{u \in V\left( G \right)} {\frac{{ - \lambda }}{{{{\left( {{d_u} - \lambda } \right)}^2}}}}$$

$$=\frac{-\lambda}{(\delta-\lambda)^2}  {\mathbf{1}^\top M{\mathbf{1}}}=\frac{-\lambda(2|E(G)|-n\lambda)}{(\delta-\lambda)^2}=\frac{-\lambda n(\overline{d} - \lambda)}{(\delta - \lambda)^{2}}.$$
By Lemma \ref{zhou1}, we have
$$\alpha(G) \leq \vartheta(G) \leq \frac{-\lambda n(\overline{d} - \lambda)}{(\delta - \lambda)^{2}}.$$
\end{proof}
The following example shows that there exists a family of graphs such that our upper bound in Theorem \ref{theta1} is smaller than the ratio-type bounds in Theorems \ref{thm1.2} and \ref{thm1.3}.
\begin{example}
Let $G_i$ be an $r_i$-regular graph with $n_i$ vertices ($i=1,2$), and let $G=G_1\vee G_2$ denote the join of $G_1$ and $G_2$, i.e., the graph obtained from $G_1\cup G_2$ by joining each vertex of $G_1$ to each vertex of $G_2$. Let $\beta_1=\beta_1(n_1,n_2,r_1,r_2)$, $\beta_2=\beta_2(n_1,n_2,r_1,r_2)$ and $\beta_3=\beta_3(n_1,n_2,r_1,r_2)$ be the ratio-type bounds on $\alpha(G)$ from Theorems \ref{theta1}, \ref{thm1.2} and \ref{thm1.3}, respectively. When $r_1,n_2$ are positive constants and $n_1$ is sufficiently large, we have
\begin{eqnarray*}
\beta_1<\beta_2<\beta_3.
\end{eqnarray*}
\end{example}
\begin{proof}
Suppose that $r_1=\lambda_1(G_1)\geq\cdots\geq\lambda_{n_1}(G_1)$ are eigenvalues of $G_1$, and $r_2=\lambda_1(G_2)\geq\cdots\geq\lambda_{n_2}(G_2)$ are eigenvalues of $G_2$. From \cite[Theorem 2.1.8]{Cvetkovic1}, we know that
\begin{eqnarray*}
\frac{r_1+r_2\pm\sqrt{(r_1-r_2)^2+4n_1n_2}}{2},\lambda_2(G_1),\ldots,\lambda_{n_1}(G_1),\lambda_2(G_2),\ldots,\lambda_{n_2}(G_2)
\end{eqnarray*}
are eigenvalues of $G=G_1\vee G_2$, and the maximum eigenvalue of $G$ is
\begin{eqnarray*}
\lambda_{max}=\frac{r_1+r_2+\sqrt{(r_1-r_2)^2+4n_1n_2}}{2}.
\end{eqnarray*}
When $r_1,n_2$ are positive constants and $n_1$ is sufficiently large, the minimum degree of $G$ is $\delta=r_1+n_2$, and the minimum eigenvalue of $G$ is $\lambda_{min}=\frac{r_1+r_2-\sqrt{(r_1-r_2)^2+4n_1n_2}}{2}$. Then
\begin{eqnarray*}
\lambda_{max}+\lambda_{min}=r_1+r_2,~\lambda_{max}\lambda_{min}=r_1r_2-n_1n_2.
\end{eqnarray*}
By computation, the upper bound in Theorem \ref{thm1.2} is
\begin{eqnarray*}
\beta_2=\frac{(n_1+n_2)(n_1n_2-r_1r_2)}{\delta^2+n_1n_2-r_1r_2}.
\end{eqnarray*}
Since the largest Laplacian eigenvalue of $G$ is $\mu=n_1+n_2$, the upper bound in Theorem \ref{thm1.3} is $\beta_3=n_1+n_2-\delta=n_1-r_1$. Then
\begin{eqnarray*}
\beta_3-\beta_2&=&\frac{\delta^2(n_1-r_1)-\delta(n_1n_2-r_1r_2)}{\delta^2+n_1n_2-r_1r_2}=\frac{\delta((r_1+n_2)(n_1-r_1)-n_1n_2+r_1r_2)}{\delta^2+n_1n_2-r_1r_2}\\
&=&\frac{\delta r_1(n_1-n_2-r_1+r_2)}{\delta^2+n_1n_2-r_1r_2}>0
\end{eqnarray*}
when $r_1,n_2$ are positive constants and $n_1$ is sufficiently large.

From the ratio-type bounds in Theorems \ref{theta1} and \ref{thm1.2}, we have $\beta_1<\beta_2$ if and only if
\begin{eqnarray*}
\overline{d}<\frac{\delta((\lambda_{max}+\lambda_{min})\delta-2\lambda_{max}\lambda_{min})}{\delta^2-\lambda_{max}\lambda_{min}}=\frac{\delta((r_1+r_2)\delta+2n_1n_2-2r_1r_2)}{\delta^2+n_1n_2-r_1r_2},
\end{eqnarray*}
where $\overline{d}=\frac{n_1(r_1+n_2)+n_2(r_2+n_1)}{n_1+n_2}$ is the average degree of $G$. Notice that
\begin{eqnarray*}
&~&\lim_{n_1\rightarrow\infty}\overline{d}=r_1+2n_2,\\
&~&\lim_{n_1\rightarrow\infty}\frac{\delta((r_1+r_2)\delta+2n_1n_2-2r_1r_2)}{\delta^2+n_1n_2-r_1r_2}=2\delta=2r_1+2n_2.
\end{eqnarray*}
Hence $\beta_1<\beta_2$ when $r_1,n_2$ are positive constants and $n_1$ is sufficiently large.
\end{proof}

\end{CJK*}
\end{spacing}
\end{document}